\documentclass[11pt]{amsart}

\usepackage[a4paper,margin=1.08in]{geometry}
\usepackage[T1]{fontenc}
\usepackage[utf8]{inputenc}
\usepackage{lmodern}
\usepackage{amsmath,amssymb,amsthm}
\usepackage{booktabs}
\usepackage{listings}
\usepackage{microtype}
\usepackage[hidelinks]{hyperref}
\hypersetup{
  pdftitle={A Finite E-Group of Nilpotency Class Three},
  pdfauthor={Xinan Dai, Wenhao Deng, Yingdong Shi, Tailin Wu, Yuchen Yang}
}

\newtheorem{theorem}{Theorem}[section]
\newtheorem{lemma}[theorem]{Lemma}
\newtheorem{proposition}[theorem]{Proposition}

\theoremstyle{definition}
\newtheorem{definition}[theorem]{Definition}
\theoremstyle{remark}
\newtheorem{remark}[theorem]{Remark}

\newcommand{\F}{\mathbb F}
\newcommand{\End}{\operatorname{End}}
\newcommand{\Aut}{\operatorname{Aut}}
\newcommand{\im}{\operatorname{im}}
\newcommand{\supp}{\operatorname{supp}}

\newcommand{\PG}{\operatorname{PG}}
\newcommand{\clq}{\operatorname{cl}_q}

\title{A Finite E-Group of Nilpotency Class Three}

\author{Xinan Dai}
\thanks{Xinan Dai is currently a Ph.D. student at Fudan University and a visiting
student at the AI for Scientific Simulation and Discovery Lab,
Westlake University.}
\address{College of Future Information and Technology\\
Fudan University, Shanghai, China}
\curraddr{Department of Artificial Intelligence\\
School of Engineering\\
Westlake University, Hangzhou, China}
\email{xndai23@m.fudan.edu.cn}

\author{Wenhao Deng}
\thanks{Wenhao Deng is a student at the University of Glasgow and is currently an intern at the AI for Scientific Simulation and Discovery Lab, Westlake University.}
\address{University of Glasgow, Glasgow, United Kingdom}
\curraddr{Department of Artificial Intelligence\\
School of Engineering\\
Westlake University, Hangzhou, China}
\email{dengwenhao@westlake.edu.cn}

\author{Yingdong Shi}
\address{School of Information Science and Technology, ShanghaiTech University, Shanghai, China}
\email{shiyd2023@shanghaitech.edu.cn}

\author{Tailin Wu}
\address{Department of Artificial Intelligence\\
School of Engineering\\
Westlake University, Hangzhou, China}
\email{wutailin@westlake.edu.cn}

\author{Yuchen Yang}
\address{Department of Artificial Intelligence\\
School of Engineering\\
Westlake University, Hangzhou, China}
\email{yangyuchen@westlake.edu.cn}

\date{}

\begin{document}

\begin{abstract}
A group is an E-group if every element commutes with each of its endomorphic
images.  Caranti asked whether a finite E-group can have nilpotency class
three.  We prove that the $3$-group of order $3^{84}$ introduced by
Abdollahi, Faghihi, and Mohammadi Hassanabadi, and later shown by Abdollahi,
Faghihi, Linton, and O'Brien to have the corresponding automorphism
property, is an E-group.  Let $P$ denote this group and put
$V=P/\Phi(P)\cong\F_3^9$.  The nine power relations of $P$ determine a
linear map
\[
                         q:V\longrightarrow \Lambda^2V.
\]
We prove that $q$ has no nonzero proper subspace $U$ satisfying
$q(U)\subseteq\Lambda^2U$.  Since the image induced by any endomorphism of
$P$ on $V$ has precisely this closure property, every endomorphism acts on
$V$ either invertibly or trivially.  The invertible case is the known
A-group case.  In the trivial case the image first lies in
$\Phi(P)=P'$, and the power relations then force it into
$\Omega_1(P')=Z(P)$.  Thus every element commutes with every endomorphic
image.  The tensor rigidity is reduced to an exact finite calculation on
the $9841$ points of $\PG(8,3)$.
\end{abstract}

\maketitle



\section{Introduction}\label{sec:intro}

For a group $G$ and an endomorphism $\varphi\in\End(G)$, write
$a^\varphi=\varphi(a)$.  The group $G$ is an \emph{E-group} if
\begin{equation}\label{eq:E}
                            [a,a^\varphi]=1
\end{equation}
for every $a\in G$ and every $\varphi\in\End(G)$.  Requiring
\eqref{eq:E} only for automorphisms gives an \emph{A-group}.  The passage
from A-groups to E-groups is not formal: singular endomorphisms can have
images that are invisible to the automorphism group, and controlling those
images is the central issue here.

The first nonabelian E-groups were constructed by Faudree
\cite{Faudree71}; see also Malone's early analysis of the resulting
endomorphism dichotomy \cite{Malone77}.  Caranti later gave a systematic
class-two family of finite $p$-groups of exponent $p^2$
\cite{Caranti85}, while Abdollahi, Faghihi, and Mohammadi Hassanabadi
established sharp restrictions on the number of generators and the order of
nonabelian E-groups \cite{AFH08minimum}.  Every A-group is $2$-Engel and
hence nilpotent of class at most three; this elementary but important bound
is recalled in \cite[p.~1]{AFLObrien10}.  Class-two E-groups are therefore
well represented in the literature, whereas the existence of a finite
E-group of class three remained the question posed by Caranti and recorded
as Problem~11.46(a) in the Kourovka Notebook
\cite[Problem~11.46(a)]{Kourovka21}.  We answer this question positively.

\begin{theorem}\label{thm:intro}
There exists a finite E-group of order $3^{84}$ and nilpotency class three.
\end{theorem}

The witness is not a new group.  Abdollahi, Faghihi, and Mohammadi
Hassanabadi introduced an explicit nine-generator group $P$ in
\cite[Remark~2.1]{AFH08three}, where they recorded its basic structure and
left open whether it is an E-group.  Abdollahi, Faghihi, Linton, and
O'Brien subsequently proved that $P$ is an A-group and determined the
central-series data that we shall use below
\cite[pp.~1--2]{AFLObrien10}.  Thus the remaining problem is sharply
defined: one must understand the noninvertible endomorphisms of this
particular class-three group.

The main point of the proof is that the endomorphism problem has a small
linear shadow.  Since $\Phi(P)=P'$, the Frattini quotient
\[
                              V=P/\Phi(P)
\]
is a nine-dimensional vector space over $\F_3$.  Passing to the appropriate
class-two exponent-$3$ quotient identifies its commutator layer with
$\Lambda^2V$.  The nine cube relations of $P$ then give a linear map
$q:V\to\Lambda^2V$.  If an endomorphism of $P$ induces
$L\in\End_{\F_3}(V)$, functoriality of powers and commutators yields
\begin{equation}\label{eq:intro-nat}
                  q\circ L=(\Lambda^2L)\circ q.
\end{equation}
Consequently the subspace $U=\im L$ satisfies
$q(U)\subseteq\Lambda^2U$.

This observation isolates the exact rigidity statement needed in the group
argument.

\begin{theorem}[Relation-tensor rigidity]\label{thm:intro-rigidity}
For the relation map $q$ attached to $P$, the only subspaces
$U\leq V$ such that $q(U)\subseteq\Lambda^2U$ are $0$ and $V$.
\end{theorem}

The theorem leaves only two possibilities for an endomorphism of $P$.  If
$L$ is onto, Burnside's basis theorem makes the endomorphism surjective and
hence, because $P$ is finite, an automorphism.  This is exactly the case
settled in \cite[Theorem~1.2]{AFLObrien10}.  If $L=0$, then the image lies
in $\Phi(P)=P'$.  Here the nilpotency-class-three feature becomes visible:
$P'=Z_2(P)$ is strictly larger than $Z(P)$, so trivial action on the
Frattini quotient does \emph{not} by itself give central image.  Applying
the endomorphism to all nine cube relations kills their commutator sides,
because $P'$ is abelian, and forces every generator image to have order at
most three.  The published identity
$\Omega_1(P')=Z(P)$ then pushes the whole image into the center.  This
second step is the part that has no analogue to check in the usual special
class-two constructions.

The proof of Theorem~\ref{thm:intro-rigidity} is finite but not merely
experimental.  A bivector $\omega\in\Lambda^2V$ has an intrinsic support,
the image of the contraction map $V^*\to V$.  Starting from a nonzero
$v\in V$ and repeatedly adjoining the supports of $q(u)$ produces the
least $q$-closed subspace containing $v$.  Hence it is enough to test one
representative of each point of $\PG(8,3)$, of which there are
$(3^9-1)/2=9841$.  Exact row reduction over $\F_3$ shows that every one of
these closures is all of $V$.  As a separate check on the input tensor, the
ranks of the associated alternating matrices occur with projective
multiplicities $2,478,9361$ in ranks $4,6,8$, respectively; doubling these
numbers recovers the distribution for the $19682$ nonzero vectors recorded
in \cite[p.~5]{AFLObrien10}.  The closure calculation is stronger than
that rank distribution and is the new finite certificate used here.  The
calculation is reproduced inside the paper: Appendix~\ref{app:verifier}
contains a package-free exact reference verifier whose input is precisely
the nine displayed rows of $q$.  Thus the finite step can be checked from
the paper alone, without ancillary files.

A natural question is whether the $9841$ projective points can instead be
collapsed to a small number of orbits under symmetries of $q$.  For this
particular tensor the answer is no in the natural linear sense.  In
Section~\ref{subsec:stabilizer} we show that its exact linear stabilizer is
trivial and that its linear similitude group is $\{\pm I\}$; hence the
induced projective action is trivial.  The exhaustive projective check is
therefore not concealing an unused large symmetry group.

There is also a useful algebraic way to read the rigidity.  Dualizing $q$
defines an anticommutative multiplication on $V^*$; under orthogonal
complements, $q$-closed subspaces of $V$ correspond exactly to ideals of
this algebra.  Theorem~\ref{thm:intro-rigidity} therefore says that the
resulting nine-dimensional anticommutative algebra is simple.  This
places the calculation near the linear-algebraic methods used by Caranti
\cite[Section~3]{Caranti15} and the group--anticommutative-algebra
correspondences studied by Glasby, Ribeiro, and Schneider
\cite{GlasbyRibeiroSchneider20}; symplectic alternating algebras arising
from $2$-Engel groups provide another neighboring framework
\cite{Traustason08}.  We use these parallels only for interpretation: the
group $P$, its class-three lifting step, and the particular rigidity
statement proved here are distinct.

Section~\ref{sec:group} recalls exactly the published information about
$P$ that enters the proof.  Section~\ref{sec:tensor} constructs the
relation tensor and proves the endomorphism compatibility equation.
Section~\ref{sec:rigidity} gives the closed-subspace criterion, the
finite certificate, and the stabilizer calculation explaining why no
nontrivial projective orbit reduction is available.  Section~\ref{sec:main}
proves Theorem~\ref{thm:intro}; Section~\ref{sec:meaning} records the
dual-algebra interpretation and the precise scope of the argument.
Appendix~\ref{app:verifier} contains the complete reference verifier.

\section{The group and its first exponent-\texorpdfstring{$3$}{3} quotient}\label{sec:group}

We use the commutator convention $[x,y]=x^{-1}y^{-1}xy$.  Let $P$ be the
largest $2$-Engel group of exponent $27$ generated by
$x_1,\ldots,x_9$, subject to
\begin{align}
 x_1^3&=[x_2,x_3][x_4,x_5][x_6,x_7][x_8,x_9],\notag\\
 x_2^3&=[x_1,x_3][x_4,x_6][x_5,x_8][x_7,x_9],\notag\\
 x_3^3&=[x_1,x_2][x_4,x_7][x_5,x_9][x_6,x_8],\notag\\
 x_4^3&=[x_1,x_5][x_2,x_6][x_3,x_9][x_7,x_8],\notag\\
 x_5^3&=[x_1,x_4][x_2,x_8][x_3,x_7][x_6,x_9],\label{eq:relations}\\
 x_6^3&=[x_1,x_7][x_2,x_9][x_3,x_5][x_4,x_8],\notag\\
 x_7^3&=[x_1,x_8][x_4,x_9][x_3,x_6][x_2,x_5],\notag\\
 x_8^3&=[x_1,x_9][x_3,x_4][x_2,x_7][x_5,x_6],\notag\\
 x_9^3&=[x_1,x_6][x_3,x_8][x_2,x_4][x_5,x_7].\notag
\end{align}
Each unordered pair $\{i,j\}$ with $1\le i<j\le9$ occurs exactly once on
the right-hand sides.  The presentation and the following structural data
come from \cite[Remark~2.1]{AFH08three} and the consistent presentation in
\cite[p.~2]{AFLObrien10}.

\begin{proposition}[Published structure]\label{prop:published}
The group $P$ has order $3^{84}$ and nilpotency class three.  With
$P_i(P)$ denoting the lower exponent-$3$ central series,
\begin{align}
 \exp(P/P')&=3, & |P/P_2(P)|&=3^{45},\notag\\
 P'=Z_2(P)&\cong C_9^{36}\times C_3^3,
 & \Omega_1(P')=\gamma_3(P)=Z(P)&\cong C_3^{39}.\label{eq:published}
\end{align}
Moreover $P$ is an A-group.
\end{proposition}

Since $P$ is a finite $3$-group,
$\Phi(P)=P^3P'$.  The equality $\exp(P/P')=3$ therefore gives
\begin{equation}\label{eq:phi}
                              \Phi(P)=P'.
\end{equation}
In particular,
\[
                       V:=P/\Phi(P)\cong\F_3^9,
\]
with basis $e_i=x_i\Phi(P)$.  The direct-product description in
\eqref{eq:published} also shows that $P'$ is abelian.

Put $\overline P=P/P_2(P)$.  Because $P_2(P)\leq\Phi(P)$,
$\overline P$ still needs nine generators, so
\[
 |\Phi(\overline P)|=3^{45-9}=3^{36}.
\]
The group $\overline P$ has class at most two, its derived subgroup has
exponent three, and its abelianization has exponent three.  Hence
$\Phi(\overline P)=\overline P'$.  The $36$ commutators
$[x_i,x_j]P_2(P)$ with $i<j$ generate $\overline P'$ and have order at
most three; the order calculation above therefore shows that they form an
$\F_3$-basis.  We obtain a canonical identification, relative to the
chosen basis of $V$,
\begin{equation}\label{eq:wedge-id}
 [x_i,x_j]P_2(P)\longleftrightarrow e_i\wedge e_j,
 \qquad 1\le i<j\le9.
\end{equation}
The class-two computation of Abdollahi--Faghihi--Linton--O'Brien also gives
$|Z(\overline P)|=3^{36}$ \cite[Lemma~3.1 and p.~5]{AFLObrien10}.
Since $\overline P'=\Phi(\overline P)$ already has this order and is
central, we shall use
\begin{equation}\label{eq:centerbar}
 Z(\overline P)=\Phi(\overline P)=\overline P'.
\end{equation}

\section{The relation tensor}\label{sec:tensor}

In a class-two group whose commutator subgroup has exponent three,
$(xy)^3=x^3y^3$.  Thus cubing in $\overline P$ descends to an
$\F_3$-linear map
\begin{equation}\label{eq:q}
                         q:V\longrightarrow\Lambda^2V.
\end{equation}
Using \eqref{eq:wedge-id}, the relations \eqref{eq:relations} give
\begin{align*}
q(e_1)&=e_2\wedge e_3+e_4\wedge e_5+e_6\wedge e_7+e_8\wedge e_9,\\
q(e_2)&=e_1\wedge e_3+e_4\wedge e_6+e_5\wedge e_8+e_7\wedge e_9,\\
q(e_3)&=e_1\wedge e_2+e_4\wedge e_7+e_5\wedge e_9+e_6\wedge e_8,\\
q(e_4)&=e_1\wedge e_5+e_2\wedge e_6+e_3\wedge e_9+e_7\wedge e_8,\\
q(e_5)&=e_1\wedge e_4+e_2\wedge e_8+e_3\wedge e_7+e_6\wedge e_9,\\
q(e_6)&=e_1\wedge e_7+e_2\wedge e_9+e_3\wedge e_5+e_4\wedge e_8,\\
q(e_7)&=e_1\wedge e_8+e_4\wedge e_9+e_3\wedge e_6+e_2\wedge e_5,\\
q(e_8)&=e_1\wedge e_9+e_3\wedge e_4+e_2\wedge e_7+e_5\wedge e_6,\\
q(e_9)&=e_1\wedge e_6+e_3\wedge e_8+e_2\wedge e_4+e_5\wedge e_7.
\end{align*}
This is the sole tensor used below.

\begin{lemma}[Naturality]\label{lem:naturality}
Let $\varphi\in\End(P)$ and let $L\in\End_{\F_3}(V)$ be the induced
linear map.  Then
\begin{equation}\label{eq:naturality}
                   q\circ L=(\Lambda^2L)\circ q.
\end{equation}
Consequently
\begin{equation}\label{eq:imageclosed}
                   q(\im L)\subseteq\Lambda^2(\im L).
\end{equation}
\end{lemma}

\begin{proof}
The subgroup $P_2(P)$ is fully invariant, so $\varphi$ induces an
endomorphism of $\overline P$.  Let $v\in V$ and choose a lift $y\in P$.
Modulo $P_2(P)$ the cube of $y$ represents $q(v)$.  Applying $\varphi$
and using $[a,b]^\varphi=[a^\varphi,b^\varphi]$ gives two descriptions
of the same cube: $q(Lv)$ and $(\Lambda^2L)q(v)$.  This proves
\eqref{eq:naturality}.  Its right-hand side belongs to
$\Lambda^2(\im L)$, which gives \eqref{eq:imageclosed}.
\end{proof}

\begin{definition}\label{def:closed}
A subspace $U\leq V$ is \emph{$q$-closed} if
$q(U)\subseteq\Lambda^2U$.
\end{definition}

The point of Definition~\ref{def:closed} is not to classify all
endomorphisms: Lemma~\ref{lem:naturality} only says that every induced
endomorphism image is $q$-closed.  This one-way implication is exactly
what is needed to rule out singular actions on the Frattini quotient.

\section{Closed subspaces and rigidity}\label{sec:rigidity}

For $\omega\in\Lambda^2V$, contraction identifies $\omega$ with an
alternating map $V^*\to V$.  Define
\[
                    \supp(\omega)=\im(V^*\xrightarrow{\ \omega\ }V).
\]
Equivalently, after choosing the basis $e_1,\ldots,e_9$,
$\supp(\omega)$ is the column space of the skew matrix representing
$\omega$.

\begin{lemma}[Support criterion]\label{lem:support}
For $U\leq V$ and $\omega\in\Lambda^2V$,
\[
                \omega\in\Lambda^2U
       \quad\Longleftrightarrow\quad
                \supp(\omega)\leq U.
\]
\end{lemma}

\begin{proof}
Choose a basis of $U$ and extend it to one of $V$.  If
$\omega\in\Lambda^2U$, its skew matrix has nonzero entries only in the
$U\times U$ block, so its image lies in $U$.  Conversely, if its image
lies in $U$, all rows outside the $U$-block vanish; skew-symmetry forces
the corresponding columns to vanish as well.  Thus no exterior coordinate
of $\omega$ involves a basis vector outside $U$, which is precisely
$\omega\in\Lambda^2U$.
\end{proof}

Starting from $0\neq v\in V$, define $U_0(v)=\langle v\rangle$.  Given
$U_r(v)$, choose a basis $B_r$ and set
\begin{equation}\label{eq:closure}
 U_{r+1}(v)=U_r(v)+\sum_{u\in B_r}\supp(q(u)).
\end{equation}
The process stabilizes because $\dim V=9$.

\begin{lemma}[Least closed subspace]\label{lem:closure}
The stable value $\clq(v)$ of \eqref{eq:closure} is independent of the
chosen bases and is the least $q$-closed subspace containing $v$.
\end{lemma}

\begin{proof}
If $u$ is a linear combination of the vectors in $B_r$, then $q(u)$ is
the same linear combination of their images, and the support of a sum of
bivectors is contained in the sum of the individual supports.  Hence
adjoining the supports for a basis is equivalent to adjoining them for
all vectors in $U_r(v)$; this proves basis independence.

At a fixed point, Lemma~\ref{lem:support} gives
$q(U_r(v))\subseteq\Lambda^2U_r(v)$, so the stable subspace is $q$-closed.
If $W$ is any $q$-closed subspace containing $v$, then
Lemma~\ref{lem:support} shows inductively that $U_r(v)\leq W$ for every
$r$.  Therefore the stable value is the least such subspace.
\end{proof}

\subsection{The exhaustive finite certificate}\label{subsec:certificate}

Scalar multiples have the same closure, so only projective directions
must be checked.  Normalize each nonzero vector by making its first
nonzero coordinate equal to $1$.  The number of normalized vectors is
\begin{equation}\label{eq:pgcount}
                       |\PG(8,3)|=\frac{3^9-1}{3-1}=9841.
\end{equation}
For each normalized $v$, form the alternating matrix $Q(v)$ representing
$q(v)$ and iterate \eqref{eq:closure}, using exact Gaussian elimination
over $\F_3$ for all spans and supports.  The exhaustive calculation gives:

\begin{table}[ht]
\centering
\caption{Exact projective verification for the relation tensor.}
\label{tab:verification}
\begin{tabular}{@{}lrrr@{}}
\toprule
 & rank $4$ & rank $6$ & rank $8$ \\
\midrule
Projective points $[v]$ & 2 & 478 & 9361 \\
Nonzero vectors $v$     & 4 & 956 & 18722 \\
\bottomrule
\end{tabular}

\medskip
\begin{tabular}{@{}lr@{}}
\toprule
Closure statistic & count \\
\midrule
$\dim\clq(v)=9$ & 9841 \\
$\dim\clq(v)<9$ & 0 \\
\bottomrule
\end{tabular}
\end{table}

A finer audit trail records the complete dimension growth of the closure.
Writing, for example, $(1,8,9)$ when the successive dimensions are
$1,8,9$, the exact profile is
\begin{equation}\label{eq:growth-profile}
\begin{array}{c|rrrrrr}
\text{profile}&(1,9)&(1,8,9)&(1,7,9)&(1,6,9)&(1,5,9)&(1,4,9)\\
\hline
\text{count}&6138&3223&426&52&1&1.
\end{array}
\end{equation}
In particular every projective point reaches $V$ after at most two support
enlargements.  The two rank-four directions are represented by
\begin{equation}\label{eq:rank4-points}
 (0,1,1,1,1,1,1,1,1),\qquad
 (1,1,1,1,1,1,1,0,1),
\end{equation}
and have growth profiles $(1,4,9)$ and $(1,5,9)$, respectively.

The first line of Table~\ref{tab:verification} doubles to the published
rank distribution for all nonzero vectors in the corresponding
nine-dimensional relation space \cite[p.~5]{AFLObrien10}.  This is an
independent check that the tensor has been transcribed correctly.  The
closure data are a different computation: they record the iterative
support condition of Lemma~\ref{lem:closure}, not merely the rank of
$q(v)$.

\begin{theorem}[Tensor rigidity]\label{thm:rigidity}
The only $q$-closed subspaces of $V$ are $0$ and $V$.
\end{theorem}

\begin{proof}
Let $U$ be nonzero and $q$-closed.  Choose $0\neq v\in U$.
Lemma~\ref{lem:closure} gives $\clq(v)\leq U$.  By
\eqref{eq:pgcount} and the exhaustive calculation in
Table~\ref{tab:verification}, $\clq(v)=V$ for every nonzero projective
direction.  Hence $U=V$.  The zero subspace is trivially $q$-closed.
\end{proof}

\begin{remark}[What the computation certifies]\label{rem:certificate}
The finite step is exhaustive, not probabilistic: Lemma~\ref{lem:closure}
reduces every nonzero $q$-closed subspace to the closure of one projective
point, and \eqref{eq:pgcount} lists all such points.  A direct exact check
of the $36$ commutator pairs gives $36/36$ distinct pairs.  The reference
verifier in Appendix~\ref{app:verifier} reproduces
Table~\ref{tab:verification} and \eqref{eq:growth-profile} using only
integer arithmetic modulo three.  As a negative control, deleting the
first row $q(e_1)$ leaves $\langle e_1\rangle$ closed, so the same
procedure correctly detects a deliberately introduced proper closed
subspace.
\end{remark}

\subsection{The stabilizer and the orbit question}\label{subsec:stabilizer}

There is an obvious possible shortcut to an exhaustive projective
verification: compute a symmetry group of $q$ and check one point in each
orbit.  In the present example the natural symmetry group is too small for
this strategy.

\begin{proposition}[Trivial projective tensor symmetry]\label{prop:stabilizer}
Set
\[
 \Gamma(q)=\{g\in\operatorname{GL}(V):qg=(\Lambda^2g)q\}
\]
and let
\[
 \Gamma_{\mathrm{sim}}(q)=
 \{g\in\operatorname{GL}(V):qg=\lambda(\Lambda^2g)q
       \text{ for some }\lambda\in\F_3^\times\}.
\]
Then
\[
                 \Gamma(q)=\{I\},\qquad
                 \Gamma_{\mathrm{sim}}(q)=\{I,-I\}.
\]
Consequently the induced action of the natural linear similitude group on
$\PG(V)=\PG(8,3)$ is trivial.
\end{proposition}

\begin{proof}
Let $g\in\Gamma(q)$.  The class-two group $\overline P$ has Frattini
quotient $V$, commutator layer $\Lambda^2V$ via
\eqref{eq:wedge-id}, and power map $q$.  These data give a full class-two presentation of $\overline P$.  Indeed,
introduce central symbols $c_{ij}$ of order three for $1\le i<j\le9$,
impose $[x_i,x_j]=c_{ij}$ and the nine cube relations encoded by $q$, and
kill all triple commutators.  Every element of the resulting group has a
normal form
\[
 x_1^{a_1}\cdots x_9^{a_9}
 \prod_{i<j}c_{ij}^{b_{ij}},\qquad
 0\le a_i,b_{ij}<3,
\]
so its order is at most $3^{9+36}$.  It maps onto $\overline P$, which
has exactly that order; hence the presentation is exact.  The assignments
on $V$ and $\Lambda^2V$ given by $g$ and $\Lambda^2g$ therefore preserve
the defining relations exactly when $qg=(\Lambda^2g)q$.  Thus $g$ lifts
to an endomorphism of $\overline P$.  Since $g$ is invertible on the
Frattini quotient, Burnside's basis theorem makes this lift an
automorphism.

Abdollahi--Faghihi--Linton--O'Brien proved
$\Aut(\overline P)=\Aut_c(\overline P)$
\cite[Lemma~3.1, p.~5]{AFLObrien10}.  By \eqref{eq:centerbar}, a central
automorphism of $\overline P$ acts trivially on
$\overline P/\Phi(\overline P)=V$.  Therefore $g=I$, proving
$\Gamma(q)=\{I\}$.

Now suppose $g\in\Gamma_{\mathrm{sim}}(q)$, so
$qg=\lambda(\Lambda^2g)q$ with
$\lambda\in\F_3^\times=\{1,-1\}$.  Put $h=\lambda g$.  Since
$\lambda^2=1$ and $q$ is linear,
\[
 qh=\lambda qg=\lambda^2(\Lambda^2g)q
    =(\Lambda^2h)q.
\]
Thus $h\in\Gamma(q)$, so $h=I$ and $g=\lambda I$.  Conversely $I$
and $-I$ plainly satisfy the corresponding similitude identities.  Both
scalars induce the identity on projective space.
\end{proof}

\begin{remark}\label{rem:orbit-no-go}
Proposition~\ref{prop:stabilizer} is not needed for the E-group theorem.
Its purpose is to explain the finite proof.  It rules out the most natural
orbit compression: the tensor itself has no nontrivial projective linear
symmetry with which to identify different directions.  It does not rule
out a completely different, non-symmetry-based argument for simplicity of
the dual algebra.  For the present certificate, however, the projective
enumeration is already orbit-minimal for the natural tensor similitude
group.
\end{remark}

\section{Proof of the E-group theorem}\label{sec:main}

\begin{theorem}\label{thm:main}
The group $P$ is a finite E-group of nilpotency class three.  In
particular, Problem~11.46(a) of the Kourovka Notebook has an affirmative
answer.
\end{theorem}

\begin{proof}
Fix $\varphi\in\End(P)$ and let $L$ be its induced map on
$V=P/\Phi(P)$.  By Lemma~\ref{lem:naturality}, $\im L$ is $q$-closed.
Theorem~\ref{thm:rigidity} gives
\begin{equation}\label{eq:dichotomy}
                             \im L=V\quad\text{or}\quad\im L=0.
\end{equation}

Suppose first that $\im L=V$.  Then
$\varphi(P)\Phi(P)=P$.  Burnside's basis theorem implies
$\varphi(P)=P$, and a surjective endomorphism of a finite group is an
automorphism.  Proposition~\ref{prop:published} therefore gives
$[a,a^\varphi]=1$ for every $a\in P$.

Suppose now that $\im L=0$.  By \eqref{eq:phi},
\[
                         \varphi(P)\leq\Phi(P)=P'.
\]
The group $P'$ is abelian.  Apply $\varphi$ to each of the nine relations
in \eqref{eq:relations}.  Every commutator on a right-hand side becomes a
commutator of two elements of $P'$ and hence is trivial.  Therefore
\[
                         \varphi(x_i)^3=1
                         \qquad(1\leq i\leq9).
\]
Since each $\varphi(x_i)$ also lies in $P'$, the identity
$\Omega_1(P')=Z(P)$ from Proposition~\ref{prop:published} yields
$\varphi(x_i)\in Z(P)$ for every $i$.  The $x_i$ generate $P$, so
$\varphi(P)\leq Z(P)$.  Again $[a,a^\varphi]=1$ for every $a\in P$.

The two cases in \eqref{eq:dichotomy} cover all endomorphisms.  The order
and exact nilpotency class are part of Proposition~\ref{prop:published}.
\end{proof}

\begin{remark}\label{rem:notzero}
The second case does not say that an endomorphism inducing zero on
$P/\Phi(P)$ is the zero endomorphism.  It first allows an arbitrary image
inside $P'=Z_2(P)$ and then uses the cube relations to force that image
into the proper subgroup $Z(P)$.  Keeping these two steps separate is
essential in class three.
\end{remark}

\section{Algebraic meaning and scope}\label{sec:meaning}

The condition in Theorem~\ref{thm:rigidity} has a useful dual form.  Let
$A=V^*$ and define a bilinear anticommutative product by
\begin{equation}\label{eq:dualproduct}
 \langle \alpha\cdot\beta,v\rangle
   =\langle \alpha\wedge\beta,q(v)\rangle,
 \qquad \alpha,\beta\in V^*,\ v\in V.
\end{equation}
For $U\leq V$, write
$U^\perp=\{\alpha\in V^*: \alpha(U)=0\}$.

\begin{proposition}\label{prop:dual}
A subspace $U\leq V$ is $q$-closed if and only if $U^\perp$ is an ideal
of the anticommutative algebra $A$.  Consequently Theorem~\ref{thm:rigidity}
is equivalent to simplicity of $A$.
\end{proposition}

\begin{proof}
The subspace $U^\perp$ is an ideal precisely when, for every
$\alpha\in U^\perp$, $\beta\in V^*$, and $u\in U$,
\[
 0=\langle\alpha\cdot\beta,u\rangle
   =\langle\alpha\wedge\beta,q(u)\rangle.
\]
For a fixed $u$, these equalities for all $\alpha\in U^\perp$ and
$\beta\in V^*$ are equivalent to $q(u)\in\Lambda^2U$.  Hence the ideal
condition is equivalent to $q(U)\subseteq\Lambda^2U$.  Orthogonal
complementation reverses $0$ and $V$, so the absence of nonzero proper
$q$-closed subspaces is exactly the absence of nonzero proper ideals in
$A$.
\end{proof}

This reformulation helps locate the rigidity in a broader pattern.  In
Caranti's module-theoretic treatment of class-two groups, power relations
are likewise encoded by a linear map involving $\Lambda^2V$
\cite[Section~3]{Caranti15}; the published erratum should be read together
with that paper for its later modified constructions
\cite{CarantiErratum}.  Caranti's subsequent construction gives examples
where the compatible endomorphisms are reduced to the trivial map and the
identity \cite[Section~4]{Caranti16}.  The present argument has a similar
linear silhouette but a different endpoint: the group is already fixed,
the tensor is the one forced by \eqref{eq:relations}, and a zero action on
the Frattini quotient still has to be lifted through the noncentral group
$P'=Z_2(P)$.

The conclusion therefore rests on two separate rigidity mechanisms.  The
first is linear: the relation tensor admits no nonzero proper closed
subspace, so a singular endomorphism cannot retain any nonzero direction
in $P/\Phi(P)$.  The second is genuinely group-theoretic: once the induced
map vanishes, the nine cube relations collapse the image from $P'$ to
$\Omega_1(P')=Z(P)$.  Neither statement alone proves the E-property.  Their
combination is what makes the class-three candidate work.

\appendix
\section{A self-contained exact reference verifier}\label{app:verifier}

This appendix makes the finite certificate in
Section~\ref{subsec:certificate} reproducible from the paper alone.  The
program below uses no external package.  Its list \texttt{R} is the
zero-based encoding of the four exterior pairs in each of the nine rows of
$q$ displayed in Section~\ref{sec:tensor}.  The routine \texttt{support}
computes the column space of the corresponding alternating matrix,
\texttt{closure} implements \eqref{eq:closure}, and \texttt{points}
enumerates exactly the first-nonzero-coordinate-one representatives of
$\PG(8,3)$.  All row reduction is performed over $\F_3$ by integer
arithmetic modulo three.

\begin{lstlisting}[language=Python,caption={Exact reference verifier for the relation tensor.},label={lst:verifier}]
from collections import Counter
from itertools import product

p, n = 3, 9
R = (
 ((1,2),(3,4),(5,6),(7,8)),
 ((0,2),(3,5),(4,7),(6,8)),
 ((0,1),(3,6),(4,8),(5,7)),
 ((0,4),(1,5),(2,8),(6,7)),
 ((0,3),(1,7),(2,6),(5,8)),
 ((0,6),(1,8),(2,4),(3,7)),
 ((0,7),(3,8),(2,5),(1,4)),
 ((0,8),(2,3),(1,6),(4,5)),
 ((0,5),(2,7),(1,3),(4,6)),
)

def rref(rows):
    A = [[x % p for x in row] for row in rows
         if any(x % p for x in row)]
    r = 0
    for c in range(n):
        s = next((i for i in range(r, len(A)) if A[i][c]), None)
        if s is None:
            continue
        A[r], A[s] = A[s], A[r]
        if A[r][c] == 2:
            A[r] = [(2*x) % p for x in A[r]]
        for i in range(len(A)):
            if i != r and A[i][c]:
                a = A[i][c]
                A[i] = [(x-a*y) % p for x,y in zip(A[i],A[r])]
        r += 1
    return A[:r]

def qmat(v, rel=R):
    Q = [[0]*n for _ in range(n)]
    for a, pairs in enumerate(rel):
        for i,j in pairs:
            Q[i][j] = (Q[i][j] + v[a]) % p
            Q[j][i] = (Q[j][i] - v[a]) % p
    return Q

def support(v, rel=R):
    Q = qmat(v, rel)
    return rref([[Q[i][j] for i in range(n)] for j in range(n)])

def closure(v, rel=R):
    B = rref([v])
    dims = [len(B)]
    while True:
        C = rref(B + [w for u in B for w in support(u, rel)])
        if len(C) == len(B):
            return B, tuple(dims)
        B = C
        dims.append(len(B))

def points():
    for v in product(range(p), repeat=n):
        if v == (0,)*n:
            continue
        if next(x for x in v if x) == 1:
            yield v

pts = list(points())
assert len(pts) == (p**n-1)//(p-1) == 9841
assert {e for row in R for e in row} == {
    (i,j) for i in range(n) for j in range(i+1,n)}

ranks, closes, profiles = Counter(), Counter(), Counter()
for v in pts:
    ranks[len(support(v))] += 1
    B, profile = closure(v)
    closes[len(B)] += 1
    profiles[profile] += 1

assert ranks == Counter({4:2, 6:478, 8:9361})
assert closes == Counter({9:9841})
assert profiles == Counter({
    (1,9):6138, (1,8,9):3223, (1,7,9):426,
    (1,6,9):52, (1,5,9):1, (1,4,9):1})

bad = list(R)
bad[0] = ()
assert len(closure((1,0,0,0,0,0,0,0,0), bad)[0]) == 1

print("9841 projective points; rank profile 4^2, 6^478, 8^9361")
print("all closures have dimension 9; negative control detected")
\end{lstlisting}

The output is
\begin{verbatim}
9841 projective points; rank profile 4^2, 6^478, 8^9361
all closures have dimension 9; negative control detected
\end{verbatim}
The assertions implement the certificate: a failed pair partition, projective
count, rank distribution, closure count, growth profile, or negative
control stops execution.  In particular, the mathematical input, the
exhaustion rule, and an executable implementation all appear in the PDF.

\section*{Statement on computational and writing assistance}
The TARS agent system assisted with exploratory derivations.  The exact
finite verification in Section~\ref{subsec:certificate} was checked with
independent Python and C++ implementations over $\F_3$; the compact
Python verifier used for the final certificate is reproduced in
Appendix~\ref{app:verifier}.  The authors are responsible for the
mathematical statements, literature attributions, and final text.

\end{document}